\documentclass [10pt,a4paper]{article}
\usepackage{amssymb}
\usepackage{enumerate,verbatim}
\usepackage{amsmath}
\usepackage{amsfonts,mathrsfs}
\usepackage{amsthm,wasysym}
\usepackage{epsfig,lpic,tikz}
\usetikzlibrary{calc}
\usepackage[applemac]{inputenc}
\usepackage[english]{babel}

\usepackage[width=0.9\textwidth]{caption}

\usepackage{pdfsync}
\usepackage{xcolor,hyperref}
\usepackage[T1]{fontenc}
\usepackage{mathtools}

\definecolor{green}{rgb}{0,0.5,0}

\hypersetup{backref,colorlinks=true,linkcolor=blue,citecolor=blue}

\usepackage{marvosym}

\newcounter{teoremaganso}
\newcounter{appendix}

\newcounter{coryganso}

\renewenvironment{abstract}{\small\quotation\noindent
 {\bfseries \abstractname .}}{\endquotation \par}

\catcode`\@=11

\def\resetthefootnote{\renewcommand{\thefootnote}{\@arabic\c@footnote} }
\def\@principiremex#1{\trivlist
 \item[\hskip \labelsep{\bfseries #1\ \thetheo.}]\ignorespaces}
\def\opar@principiremex#1[#2]{\trivlist
 \item[\hskip \labelsep{\bfseries #1\ \thetheo\ (#2).}]\ignorespaces}

\newcommand{\newTHEOremrom}[2]{\newenvironment{#1}{\refstepcounter{theo}\@ifnextchar[{\opar@principiremex{#2}}
{\@principiremex{#2}}}{\qedB\endtrivlist}} \catcode`\@=12
\DeclareMathSymbol{\square}{\mathord}{AMSa}{"03}
\newcommand{\qedB}{\nopagebreak\hspace*{\fill}$\square$\par}
\newcommand{\Qed}{\nopagebreak\hspace*{\fill}{\vrule width6pt height6pt depth0pt}\par}

\newtheorem {bigtheo} [teoremaganso] {Theorem}

\newTHEOremrom {defi} {Definition}
\newTHEOremrom {obs} {Remark}
\newTHEOremrom {ex} {Example}

\newcommand{\N}{\ensuremath{\mathbb{N}}}

\newcommand{\R}{\ensuremath{\mathbb{R}}}

\newcommand{\F}{\ensuremath{\mathcal{F}}}

\newcommand{\mr}[1]{\mathrm{#1}}

\title{\textbf{Fake saddles and their transition maps}
\footnotetext{2020 {\it AMS Subject Classification}: 34D10, 34D20; Secondary: 34C05.} 
\footnotetext{{\it Key words and phrases}: singularities, Poincar\'e and Dulac transition maps, asymptotic expansion, uniform flatness, stability}
\footnotetext{This work has been partially funded by the Ministry of Science, Innovation and Universities of Spain through the grant PID2021-125625NB-I00  and by the Agency for Management of University and Research Grants of Catalonia through the grant 2021SGR01015.
}
}

\author{David Mar\'{\i}n\footnote{david.marin"AT"uab.cat}
\\*[.1truecm]
{\small \textsl{Departament de Matem{\`a}tiques, Edifici Cc,
Universitat Aut{\`o}noma de Barcelona,}}\\*[-.05truecm]
{\small\textsl{08193 Cerdanyola del Vall\`es (Barcelona), Spain}}
}

\date{\vspace{-5ex}}

\begin{document}

\maketitle

\begin{abstract}
We study degenerate singular points of planar vector fields  inside a (degenerate) flow box.
These kind of singularities are called fake saddles and their linear parts are always zero.
We characterize fake saddles with non-zero second-order jet and we give the first term of a uniform asymptotic expansion of the Poincar\'e map between two transverse sections to their corresponding singular fiber, determining its stability.
\end{abstract}

\section{Introduction and statements of main results}

Following \cite{NZ,PRT15} we  define a \emph{fake saddle} as a singular point having exactly two separatrices which are contained in a smooth invariant curve separating two hyperbolic sectors, see Figure~\ref{fig1}. This kind of singularities are also known as  impassable grains.
Another way to think about a fake saddle is that the corresponding vector field can be put in a form that is like a degenerate flow box, i.e. near the singularity the phase portrait consists of parallel fibers, all but one of which have no singular points, and the singular fiber (the union of the singular point with its two separatrices) has a semi-stable equilibrium point, see \cite{CGP25}. Taking two transverse sections $\Sigma_\alpha$ and $\Sigma_\omega$ to the singular fiber outside the singular point there is a transition map $\Pi_\alpha^\omega:\Sigma_\alpha\to\Sigma_\omega$ which is well-defined on both sides of the singular fiber.

The local phase portrait of a singularity with non-zero linear part (for the nilpotent case see for instance \cite[Theorem~3.5]{DLA}) is well known and prevents the singular point from being a fake saddle. Notice that the smoothness of the singular fiber is strictly necessary because the Hamiltonian vector field $y\partial_x+x^2\partial_y$ has a nilpotent singularity at the origin with exactly two separatrices that are contained in the cuspidal curve $y^2-\frac{2}{3}x^3=0$ separating two hyperbolic sectors.

The first non-trivial and generic case of a fake saddle arises when the second-order jet is non-zero. Our objective is to characterize these generic fake saddles and to analyze their transition maps in order to determine whether the behavior is attractive or repulsive on each side of the singular fiber. The techniques we employ allow us to derive the leading term of an asymptotic expansion that is uniform with respect to parameters.

\begin{figure}[t]
\begin{center}
\includegraphics{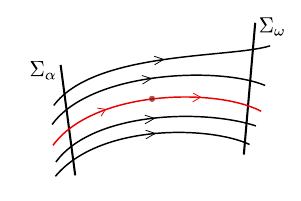}

\caption{Fake saddle and the associated transition maps.}\label{fig1}
\end{center}
\end{figure}

%

This work is mainly motivated by the paper \cite{CGP25}, which provides a normal form for generic fake saddles that we recall in the sequel. First we can locate the fake saddle at the origin $(0,0)$. Secondly, if the vector field is at least of class $\mathscr C^3$ there exists a local diffeomorphism that rectifies the germ of singular fiber, so that it is contained in the line $y=0$ in a certain coordinate system $(x,y)$, see~\cite[Lemma 2.1]{CGP25}. Thanks to the hypothesis that the second-order jet of the fake saddle is non-zero, after a suitable rescaling of the coordinates we can assume that the fake saddle is defined by a system of differential equations
\[\left\{\begin{array}{l}
\dot x=x^2+axy+y^2+\mr O_3(x,y),\\ \dot y=y\big(cx+by+\mr O_2(x,y)\big),\end{array}\right.\]
where $a,b,c\in\R$. 
  This normal form motivates us to consider a smooth family of planar vector fields $\{X_\mu\}_{\mu\in\Omega\subset\R^n}$ having the form
 \begin{equation}\label{X}
 X_\mu(x,y)=(x^2f_1(x,y;\mu)+a(\mu)xy+y^2f_2(x,y;\mu))\partial_x+\big(xg_1(x,y;\mu)+yg_2(y;\mu)\big)y\partial_y,
 \end{equation}
 where $f_1(x,y;\mu)$, $f_2(x,y;\mu)$, $g_1(x,y;\mu)$ and $g_2(y;\mu)$ are $\mathscr C^K$-functions fulfilling $f_1(0,0;\mu)=f_2(0,0;\mu)=1$.
 We consider the following invariants 
\begin{equation}\label{abc}
a(\mu),\quad  b(\mu):=g_2(0;\mu),\quad c(\mu):=g_1(0,0;\mu)
\end{equation}
of the family (\ref{X})
and define the associated value $d(\mu):=4(1-c(\mu))-(a(\mu)-b(\mu))^2$, which will play a key role in the sequel.

In this paper we  extend the main theorems of \cite{CGP25} and we give simpler proofs of them using the results stated in \cite{MV21,MV24}.
 Our first result gives a characterization of generic fake saddles inside the family (\ref{X}), generalizing \cite[Theorem~A]{CGP25} which only treats the case $a=0$.
\begin{bigtheo}\label{car}
If the invariants $(a,b,c)$ given in (\ref{abc}) of a vector field $X$ in the family (\ref{X}) do not belong to $\{d=0\}\cap\{a^2-b^2=4\}$ then
the origin is a fake saddle
if and only if, either $d>0$, or $c=1$ and $a=b$. In both situations after blowing up the origin we have a single singular point on the exceptional divisor, which is a hyperbolic saddle of hyperbolicity ratio $1-c>0$ in the first one and a semi-hyperbolic saddle in the second one. 
\end{bigtheo}

The hypothesis $(a,b,c)\notin\{d=0\}\cap\{a^2-b^2=4\}$ in Theorem~\ref{car} (which is always verified when $a=0$) can not be removed as Example~\ref{hyp} will show.

Let $\Omega$ be an open subset of $\R^n$ and let $W$ any subset of $\Omega$. We recall (cf. \cite[Definition 1.2]{MV21}) that a $\mathscr C^K$-function $f(s;\mu)$ defined in the intersection of $(0,+\infty)\times\Omega$ with an open neighborhood of $\{0\}\times\Omega\subset\R^{n+1}$  belongs to the flat class $\F_L^K
(W)$ if for each $\mu_0\in W$ and $\nu=(\nu_0,\ldots,\nu_n)\in\N^{n+1}$ with $|\nu|=\nu_0+\cdots+\nu_n\le K$ 
there exists $\varepsilon>0$ such that $|\partial_s^{\nu_0}\partial_{\mu_1}^{\nu_1}\cdots\partial_{\mu_n}^{\nu_n} f(s;\mu)|\le C s^{L-\nu_0}$ for every $s\in(0,\varepsilon)$ and $\mu\in\mathbb B_\varepsilon(\mu_0)\cap\Omega$. 

We consider the transverse sections $\Sigma_\star=\{x=\star\}$ to $y=0$, for $\star\in\{\alpha,\omega\}$, where $\alpha<0<\omega$. We assume that $f_1(x,0)>0$ for all $x\in[\alpha,\omega]$ and we consider the transition map $\Pi_\alpha^\omega:\Sigma_\alpha\to\Sigma_\omega$. Our second result extends \cite[Theorem~B]{CGP25}, which only treats the case $a=b=0$, $\partial_yf_1=\partial_y g_1=0$, $f_2=1$ and $g_2=0$, when the transition map is  restricted to $y>0$. In that theorem
the authors characterize the stability of the transition map from the sign of  the Cauchy principal value
\[\mr{PV}\int_\alpha^\omega\frac{g_1(x,0;\mu)}{xf_1(x,0;\mu)}dx:=\lim\limits_{\varepsilon\to 0^+}\left(\int_\alpha^{-\varepsilon}\frac{g_1(x,0;\mu)}{xf_1(x,0;\mu)}dx+\int_{\varepsilon}^\omega\frac{g_1(x,0;\mu)}{xf_1(x,0;\mu)}dx\right).\]

\begin{bigtheo}\label{coef}
Assume that $\mu_0\in\Omega$ and that the invariants (\ref{abc}) of the vector field $X_{\mu_0}$ in (\ref{X})  belong to  $\{d=4(1-c)-(a-b)^2>0\}$. If $\mu\approx\mu_0$ then the origin is a fake saddle of $X_\mu$ and its transition map 
 $\Pi_\alpha^\omega:\Sigma_\alpha\to\Sigma_\omega$ satisfies $\Pi_\alpha^\omega(y;\mu)=e^{\gamma_\pm(\mu)} y+\F_{1+\epsilon}^K(\{d>0\})$ on $\pm y\ge 0$,
where $\epsilon>0$ and
\[\gamma_\pm(\mu)=\mr{PV}\int_\alpha^\omega\frac{g_1(x,0;\mu)}{xf_1(x,0;\mu)}dx\pm\frac{\pi(2b(\mu)-c(\mu)(a(\mu)+b(\mu)))}{\sqrt{d(\mu)}}.\]
\end{bigtheo}
\begin{obs}\label{PV}
Using that $f_1(0,0)=1$ and $g_1(0,0)=c$ we can explicitly compute the previous Cauchy principal value using a convergent integral:
\[\mr{PV}\int_\alpha^\omega\frac{g_1(x,0)}{xf_1(x,0)}dx
=c\log\left|\frac{\omega}{\alpha}\right|+\int_\alpha^\omega\left(\frac{g_1(x,0)}{f_1(x,0)}-c\right)\frac{dx}{x}.\]
\end{obs}
\begin{obs}
Clearly the sign of $\gamma_\pm(\mu)$ determines the stability of the singular fiber $y=0$ on the side $\pm y>0$. In the case $a=b=0$ the two values $\gamma_\pm(\mu)$ coincide with the principal value of the integral. Otherwise the stability of the two sides $\pm y>0$ can be different, see Example~\ref{ex6}.
\end{obs}
It is worth to be noticed that the writing of $\Pi_\alpha^\omega(y;\mu)$ in the statement  implies that it is of the form $e^{\gamma_\pm(\mu)}y+\mr{o}(y)$ but also that the remainder term $\mr{o}(y)$ is uniform with respect to the parameter $\mu$.  This uniformity allows to address cyclicity problems and not just study stability. For clarity in the exposition we will omit the dependence on $\mu$ when it is not essential.

\section{Proofs of the main results}

\begin{proof}[Proof of Theorem~\ref{car}]
A straightforward computation shows that
\[Y=\frac{(u,uv)^*X}{u}=P(u,v)u\partial_u+Q(u,v)v\partial_v\]
with $P(0,0)=1$ and $Q(0,v)=-v^2+(b-a)v+c-1$. The point $(u,v)=(0,0)$ is always a singular point of $Y$.  If $d<0$ then $Y$ has two other singular points on the exceptional divisor $u=0$ with at least one non-zero eigenvalue.  If $(a,b,c)\in\{d=0\}\setminus(\{c=1,a=b\}\cup\{a^2-b^2=4\})$ then $Y$ has another double singular point at $(u,v)=(0,(b-a)/2)$ with a non-zero eigenvalue. In these situations the transition map $\Pi_\alpha^\omega:\{x=\alpha\}\to\{x=\omega\}$ is not well defined so that the origin is not a fake saddle of $X$. The remaining assertions are easy to check.
 \end{proof}

\begin{proof}[Proof of Theorem~\ref{coef}]
Let us compute first the value $\gamma_+$, i.e. we assume that $y\ge 0$.
Consider the charts 
 $(x,y)=\pi_\pm(u,v)=(\pm u(1-v),uv)$ of the blow-up of the origin, see Figure~\ref{B1}.
 
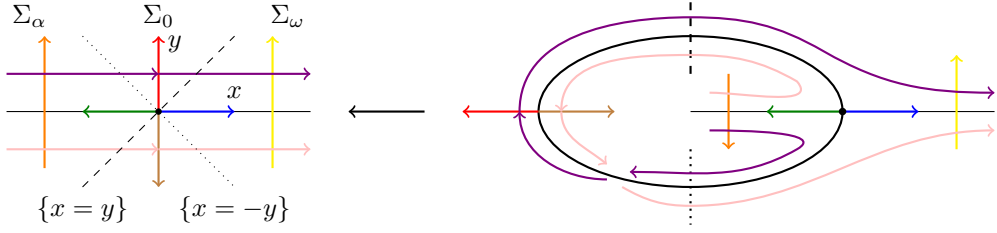
\begin{figure}[t]
\begin{center}
\begin{tikzpicture}
\begin{scope}[xshift=-4cm]
\draw (-2,0) to (2,0);
\draw[->,orange,thick] (-1.5,-0.75) to (-1.5,1);
\draw[->,yellow,thick] (1.5,-0.75) to (1.5,1);
\draw[->,blue,thick] (0,0) to (1,0);
\draw[->,red,thick] (0,0) to (0,1);
\draw[->,green,thick] (0,0) to (-1,0);
\draw[->,brown,thick] (0,0) to (0,-1);
\draw[fill] (0,0) circle [radius=1pt];
\draw [violet,thick,->] (-2,.5) to (0,.5);
\draw[violet,thick,->] (0,.5) to (2,.5);
\draw [pink,thick,->] (-2,-.5) to (0,-.5);
\draw[pink,thick,->] (0,-.5) to (2,-.5);
\draw[dashed] (-1,-1) to (1,1);
\draw[dotted] (-1,1) to (1,-1);
\node at (-1.7,1) [anchor=south] {$\Sigma_\alpha$};
\node at (0,1) [anchor=south] {$\Sigma_0$};
\node at (1.7,1) [anchor=south] {$\Sigma_\omega$};
\node at (-1,-1) [anchor=north] {$\{x=y\}$};
\node at (1,-1) [anchor=north] {$\{x=-y\}$};
\node at (1,0.05) [anchor=south] {$x$};
\node at (0,.9) [anchor=west] {$y$};
\end{scope}
\draw[<-,thick] (-1.5,0) to (-0.5,0);
\begin{scope}[xshift=3cm,thick]
\draw (0,0) ellipse [x radius=2,y radius=1];
\draw[thin] (0,0) to( 4,0);
\draw[blue,->] (2,0) to(3,0);
\draw[green,->](2,0) to (1,0);
\draw[red,->] (-2,0) to(-3,0);
\draw[brown,->] (-2,0) to (-1,0);
\draw[orange,->] (.5,.5) to (.5,-0.5);
\draw[yellow,->] (3.5,-.5) to (3.5,0.75);
\draw[dotted] (0,-0.5) to (0,-1.5);
\draw[dashed] (0,.5) to (0,1.5);
\draw[violet,->]  (.25,-0.25) to [out=0,in=30]  (1.4,-.5) to [out=210,in=0] (-0.8,-0.8);
\draw[violet,->] (-1.1,-0.9) to [out=180,in=-90] (-2.25,0) ;
\draw[violet,->](-2.25,0) to [out=90,in=180] (0,1.25) to [out=0,in=150] (1.75,.8) to [out=-30,in=180] (4,.25);
\draw[pink,->] (.25,.25) to [out=0,in=-50] (1.4,.4) to [out=150, in=0] (0,.75) to [out=180,in=90] (-1.7,0);
\draw[pink,->] (-1.7,0) to [out=-90,in=140] (-1.1,-.7);
\draw[pink,->] (-.9,-1) to [out=-30,in=180] (0,-1.25)  to [out=0,in=180] (4,-0.25);
\draw[fill] (2,0) circle [radius=1pt];
\end{scope}
\end{tikzpicture}
\end{center}
\caption{Blowup of the fake saddle at the origin.}\label{B1}
\end{figure}

\noindent Notice that $v=\frac{y}{y\pm x}$ is a coordinate on the exceptional divisor $u=0$.
 We assume that $u,v\ge 0$, we write
\[\frac{\pi_\pm^*X}{u}=P_\pm(u,v)u\partial_u+Q_\pm(u,v)v\partial_v\]
and define
\[X_\pm(x,y)=P_1^\pm(x,y)x\partial_x+P_2^\pm(x,y)y\partial_y\]
with
\[P_1^-(x,y)=Q_-(y,x),\quad P_2^-(x,y)=P_-(y,x),\quad P_1^+(x,y)=P_+(x,y),\quad P_2^+(x,y)=Q_+(x,y).\]
It can be checked that $\lambda_\pm=-\frac{P_2^\pm(0,0)}{P_1^\pm(0,0)}>0$, so that $(0,0)$ is a hyperbolic saddle of $X_\pm$.
Consider the Dulac map $D_\pm$ of $X_\pm$ between the parametrized transverse sections
\begin{equation}\label{sigma}
\sigma_1^-(s)=\Big(\frac{s}{-\alpha+s},-\alpha+s\Big),\quad\sigma_2^-(s)=(1,s),\quad \sigma_1^+(s)=(s,1),\quad \sigma_2^+(s)=\Big(\omega+s,\frac{s}{\omega+s}\Big).
\end{equation}
We know from \cite[Theorem A]{MV21} that $D_\pm(s;\mu)=s^{\lambda_\pm(\mu)}(\Delta_{00}^\pm(\mu)+\F_\epsilon)$ for some $\epsilon>0$, where we write $\F_\epsilon$ instead of $\F_\epsilon^K(\{\lambda_\pm>0\})$ for brevity. 
Then, using \cite[Lemma A.2]{MV21}, we can write
\[\Pi_\alpha^\omega(y)=(D_+\circ D_-)(y)=\Big(y^{\lambda_-}(\Delta_{00}^-+\F_\epsilon)\Big)^{\lambda_+}(\Delta_{00}^++\F_\epsilon)=\Delta_{00}y+\F_{1+\epsilon},\]
with 
$\Delta_{00}=(\Delta_{00}^-)^{\lambda_+}\Delta_{00}^+$. Following \cite{MV24}, 
to compute the coefficients $\Delta_{00}^\pm$ we must introduce some auxiliary functions:
\begin{align*}
R_{12}^-(y)&=\frac{P_1^-}{P_2^-}(0,y)=\frac{g(-y,0)}{f(-y,0)}-1,\\
R_{21}^-(x)&=\frac{P_2^-}{P_1^-}(x,0)=\frac{1}{1-c}\left[\frac{\left(-a +b +c -2\right) x^{2}+\left(a -c +2\right) x-1}{\left(a -b -c +2\right) x^{2}+\left(-a +b +2 c -2\right) x +1-c}\right],\\
R_{12}^+(y)&=\frac{P_1^+}{P_2^+}(0,y)=\frac{1}{1-c}\left[\frac{\left(-a +b -c +2\right) y^{2}+\left(a +c -2\right) y+1}{\left(a -b +c -2\right) y^{2}+\left(-a +b -2 c +2\right) y +c -1}\right],\\
R_{21}^+(x)&=\frac{P_2^+}{P_1^+}(x,0)=\frac{g(x,0)}{f(x,0)}-1.
\end{align*}
Notice that $\lambda:=\lambda_+=1-c>0$ and 
and $\lambda_-=\frac{1}{\lambda_+}$. According to \cite[p. 47]{MV24}
we consider the functions
\begin{align*}
\log L_1^-(u)&:=\int_0^u\left(R_{12}^-(y)+\frac{1}{\lambda_-}\right)\frac{dy}{y}=\int_0^{-u}\left(\frac{g(x,0)}{f(x,0)}-c\right)\frac{dx}{x},\\
\log L_2^-(u)&:=\int_0^u\left(R_{21}^-(x)+{\lambda_-}\right)\frac{dx}{x}=\frac{1}{1-c}\int_0^u\frac{\Big(c \left(a -b -c +2\right) x-\left(a c -c^{2}-b +c \right) \Big)dx}{\left(a -b -c +2\right) x^{2}-\left(a -b -2 c +2\right)  x -\left(c -1\right)},\\
\log L_1^+(u)&:=\int_0^u\left(R_{12}^+(y)+\frac{1}{\lambda_+}\right)\frac{dy}{y}=\frac{1}{1-c}\int_0^u \frac{\Big(c \left(a -b +c -2\right) y-\left(a c +c^{2}-b -c \right) \Big)dy}{\left(a -b +c -2\right) y^2-\left(a -b +2 c -2\right)  y +\left(c -1\right)},\\
\log L_2^+(u)&:=\int_0^u\left(R_{21}^+(x)+{\lambda_+}\right)\frac{dx}{x}=\int_0^{u}\left(\frac{g(x,0)}{f(x,0)}-c\right)\frac{dx}{x}.
\end{align*}
From (\ref{sigma}) we compute the partial derivatives $\sigma_{ijk}^\pm=\partial_s^k\sigma_{ij}^\pm(0)$ of the parametrizations $\sigma_i^\pm(s)=(\sigma_{i1}^\pm(s),\sigma_{i2}^\pm(s))$:
\[\sigma^-_{111}=-1/\alpha,\quad \sigma^-_{120}=-\alpha,\quad \sigma^-_{210}=1,\quad \sigma^-_{221}=1,\quad \sigma_{111}^+=1,\quad \sigma_{120}^+=1,\quad \sigma_{210}^+=\omega,\quad \sigma_{221}^+=1/\omega.
\]
Then, according to \cite[Theorem~A]{MV24} and Remark~\ref{PV}, we obtain the following expression:
\begin{align*}\Delta_{00}&=(\Delta_{00}^-)^{\lambda_+}\Delta_{00}^+=\frac{\sigma_{111}^-(\sigma_{120}^-)^\lambda (L_2^-(\sigma_{210}^-))^{\lambda}(\sigma_{111}^+)^{\lambda}\sigma_{120}^+L_2^+(\sigma_{210}^+)}{L_1^-(\sigma_{120}^-)(\sigma_{221}^-)^{\lambda}\sigma_{210}^-(L_1^+(\sigma_{120}^+))^\lambda\sigma_{221}^+(\sigma_{210}^+)^\lambda}=\frac{(-\alpha)^{-1+\lambda}L_2^+(\omega)}{\omega^{-1+\lambda}L_1^-(-\alpha)}\left(\frac{L_2^-(1)}{L_1^+(1)}\right)^\lambda\\
&=\exp\left( \mr{PV}\int_\alpha^\omega\frac{g_1(x,0)}{xf_1(x,0)}dx+\gamma_0\right),
\end{align*}
where $\gamma_0:=\lambda\log\left(\frac{L_2^-(1)}{L_1^+(1)}\right)$.
Let us prove now that 
\[\gamma_0=-\frac{\pi(c(a+b)-2b)}{\sqrt{d}}.\]
With this aim notice that
\[\log L_1^+(u;a,b,c)=\log L_2^-(u;-a,-b,c)=\alpha(u;a,b,c)+\beta(u;a,b,c)\]
where
\[\alpha(u;a,b,c)=\frac{c}{2(1-c)}\log\left(\frac{1-c+(a-b+2c-2)u+(-a+b-c+2)u^2}{1-c}\right)\]
and
\[\beta(u;a,b,c)={\scriptstyle-\frac{((a+b)c-2b)}{(1-c)\sqrt{d}}}\left[\arctan\left(\frac{2(1-c)+b-a+2(a-b+c-2)u}{\sqrt{d}}\right)-\arctan\left(\frac{2(1-c)+b-a}{\sqrt{d}}\right)\right].\]
Since $\alpha(1;a,b,c)=-\frac{c}{2(1-c)}\log(1-c)$ 
we obtain that 
\[
\log\left(\frac{L_2^-(1)}{L_1^+(1)}\right)=\beta(1;-a,-b,c)-\beta(1;a,b,c)=\frac{((a+b)c-2b)}{(1-c)\sqrt{d}}F(a,b,c),\]
where $F(a,b,c)$ is the following function
\[\arctan\left(\frac{b-a-2}{\sqrt{d}}\right)-\arctan\left(\frac{b-a+2-2c}{\sqrt{d}}\right)+\arctan\left(\frac{-b+a-2}{\sqrt{d}}\right)-\arctan\left(\frac{-b+a+2-2c}{\sqrt{d}}\right).\]
Notice that $F(a,b,c)=G(c,e)$ only depends on $c$ and $e=b-a$ because $d=4-4c-e^2$. It can be checked that $\partial_cG=\partial_eG=0$ so that $G$, and consequently $F$, is constant. Evaluating at $e=0$ we obtain that $F\equiv -\pi$.

To compute $\gamma_-$ it suffices to apply  to $X$ the symmetry $(x,y)\mapsto (x,-y)$ that transforms the invariants $(a,b,c)$ into $(-a,-b,c)$ and the value $\gamma_0$ into $-\gamma_0$.
\end{proof}

  \section{Examples and applications}

 \begin{ex}\label{hyp}
For each integer $n\ge 3$ we consider the vector field 
 \[X_n=(x+y)^2\partial_x+y^n\partial_y\]
 having a degenerate singularity at the origin
with invariants $(a,b,c)=(2,0,0)\in\{d=0\}\cap\{a^2-b^2=4\}$. 

It can be checked that in the resolution of singularities of $X_3$ appears a saddle-node whose weak separatrix is not the strict transform of $y=0$  and meets transversely the exceptional divisor. Consequently the origin is not a fake saddle for the vector field $X_3$. 
  
 On the other hand, the blowup of the vector field $X_4$ 
\[Y_0=\frac{(v,uv)^*X_4}{v}=(-(u+1)^2+u^3v^2)u\partial_u+(u+1)^2v\partial_v\]
has two singular points  on the exceptional divisor $v=0$: $(u,v)=(0,0)$ which is a hyperbolic saddle and $(u,v)=(-1,0)$ which is degenerated. Moreover,
\[Y_1=(u-1,v)^*Y_0=(u^2+v^2-u^3-4uv^2+6u^2v^2-4u^3v^2+u^4v^2)\partial_u+u^2v\partial_v\]
has invariants $(a,b,c)=(0,0,0)$, so that $d=4$ and the transition map along $v=0$ is well-defined for~$Y_1$. Thus, the origin is a fake saddle of $X_4$, see Figure~\ref{X0}.
This example (for which $a_{1,1}:=a=2$, $h_{0,1}:=b=0$ and $h_{1,0}:=c=0$) contradicts the following claim of~\cite{CGP25}:
``it is readily seen that a necessary condition for the origin to be a fake singularity on the singular fiber $y = 0$ is that [...] either $(h_{0,1} -a_{1,1})^2 + 4(h_{1,0} - 1) <0$, or that $h_{1,0} = 1$ and $h_{0,1} = a_{1,1}$.''
\begin{figure}[t]
\begin{center}
\includegraphics[width=8cm]{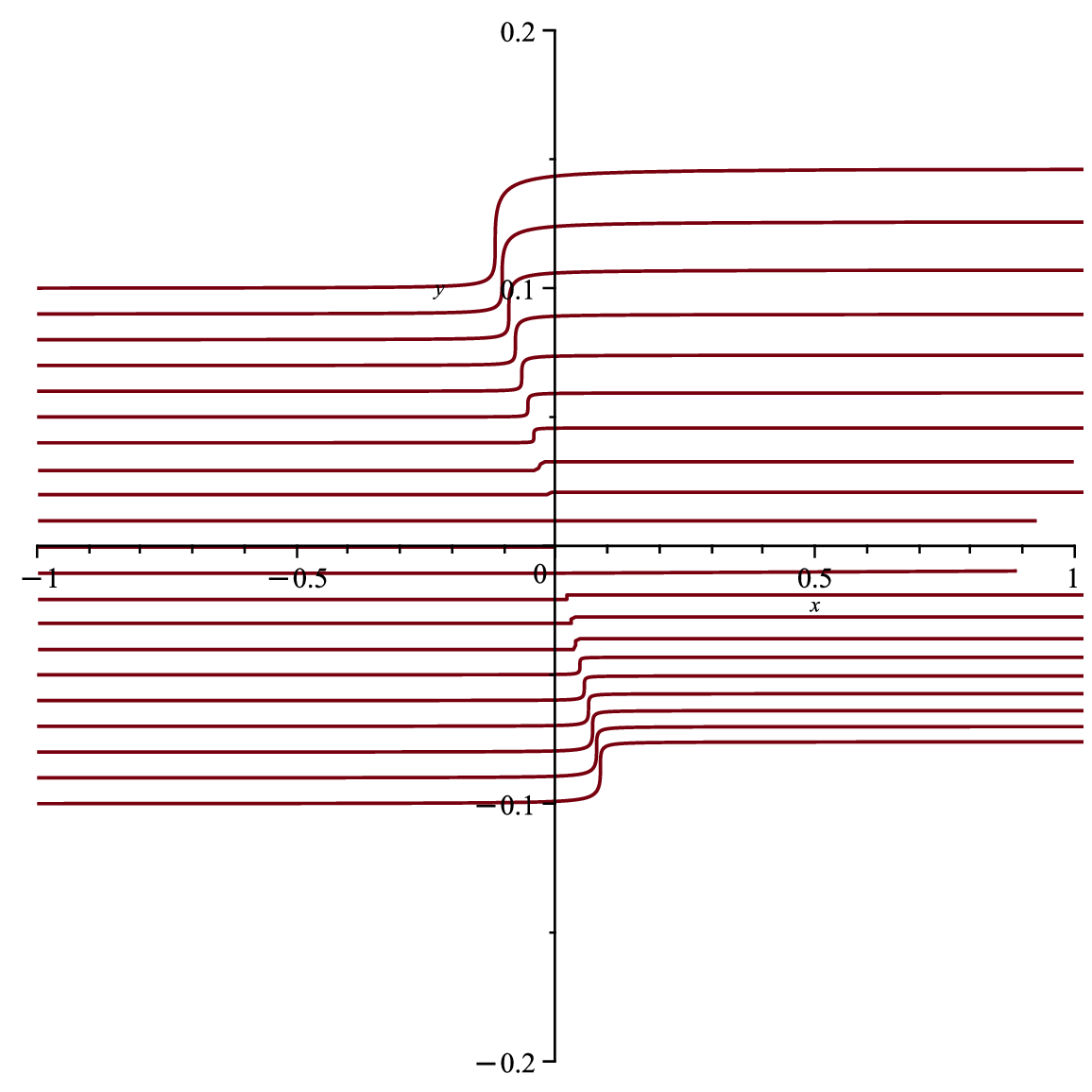}
\end{center}
\caption{Numerical solutions of $X_4=(x+y)^2\partial_x+y^4\partial_y$.}\label{X0}
\end{figure}

\noindent Moreover,  since $a=b=0$, according to Theorem~\ref{coef}, the derivative at $v=0$ of the transition map $\Pi_\alpha^\omega(v)$ of $Y_1$ is
\[\exp\mr{PV}\int_\alpha^\omega\frac{u^2}{u^2-u^3}du=\left|\frac{1-\alpha}{1-\omega}\right|,\]
so that it is contractive or repulsive on both sides $\pm v>0$. Notice that the transition map $\Pi_{-1}^{+1}(y)$ of $X_4$ is contractive on one side $y>0$ and repulsive on the other side $y<0$, see Figure~\ref{X0}. 
 \end{ex}

\begin{figure}[h]
\begin{center}
\includegraphics[width=5.5cm]{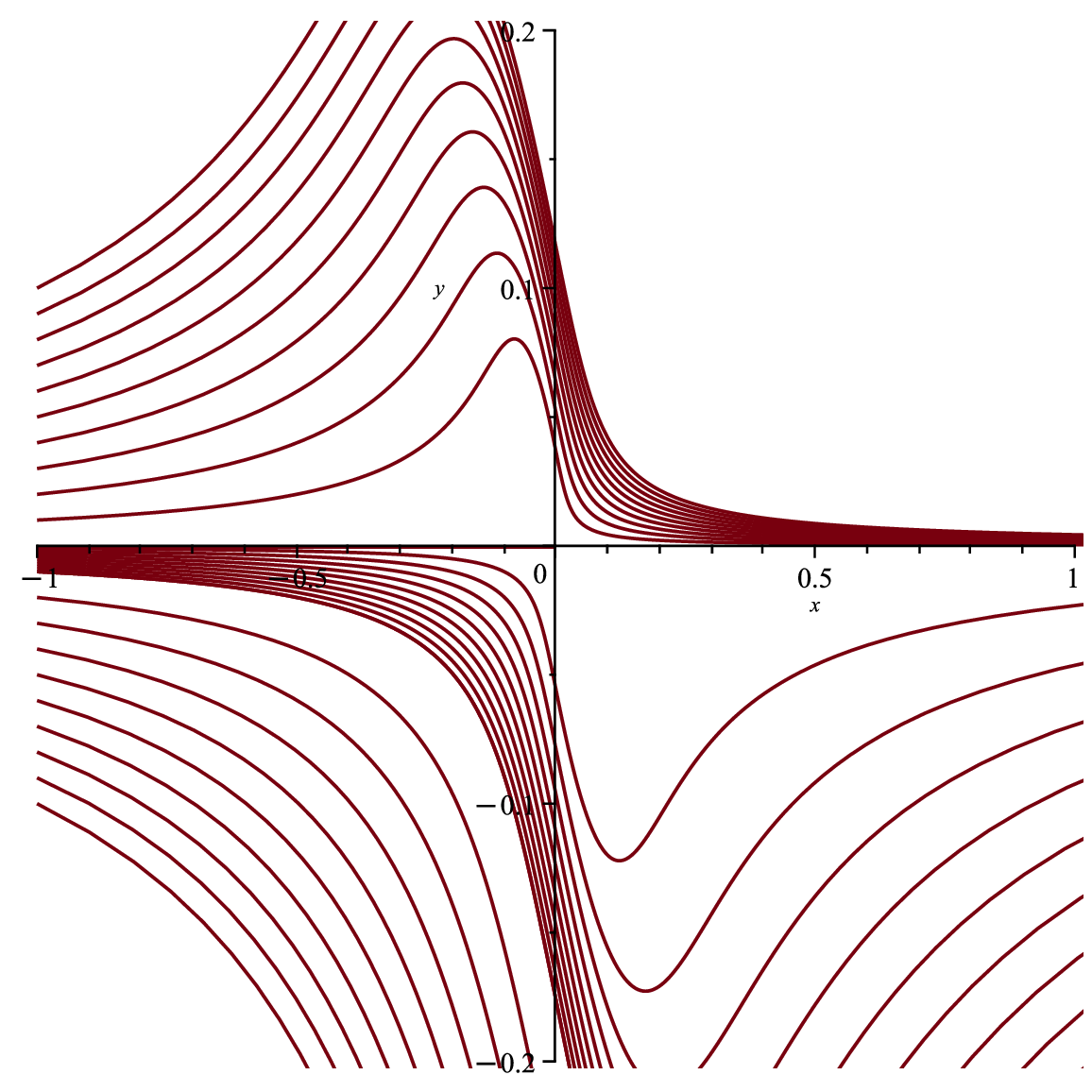}
\end{center}
\caption{Local phase portrait of the vector field
$(x^2+y^2+xy)\partial_x-(x+y)y\partial_y$
having first integral $\ln \! \left(y^2(2 x^{2}+2 x y +y^{2})\right)-2 \arctan \! \left(\frac{x +y}{x}\right)$.}\label{num}
\end{figure}

\begin{ex}\label{ex6}
Consider the quadratic homogeneous vector field $X=(x^2+y^2+axy)\partial_x+(cx+by)y\partial_y$ with $d=4(1-c)-(a-b)^2>0$. Taking $\alpha=-1$ and $\omega=1$ we have that
$\mr{PV}\int_\alpha^\omega\frac{g_1(x,0)}{f_1(x,0)}\frac{dx}{x}=\mr{PV}\int_{-1}^1\frac{cx}{x^2}dx=0$. Taking also  $a=1$, $b=-1$, $c=-1$ (so that $d=4$)  the transition map  $\Pi_{-1}^{+1}$ is contractive for $y>0$ and expansive for $y<0$ in accordance with 
 $\Pi_{-1}^{+1}(y)=e^{\gamma_\pm}y+\F_{1+\epsilon}$, where $\gamma_\pm=0\pm\pi\frac{(-2(-1)-(1-1)(-1))}{2}=\pm\pi$ depending on the sign $\pm$ of $y$, see Figure~\ref{num}.
\end{ex}

To finish this work we apply Theorems~\ref{car} and~\ref{coef} to the study of the following family of  degenerated singularities considered in~\cite{GMM02}:
\begin{equation}\label{Z}
Z_\mu:\left\{\begin{array}{l}\dot x=\beta x^2y+\alpha xy^2-\beta y^3-x^4,\\
\dot y=4\beta xy^2+\alpha y^3+2x^5,
\end{array}\right.\qquad \mu=(\alpha,\beta)\in\R^2.
\end{equation}
According to \cite[p. 189]{GMM02}, the origin is monodromic for $Z_\mu$ if and only if $\beta>1/4$.

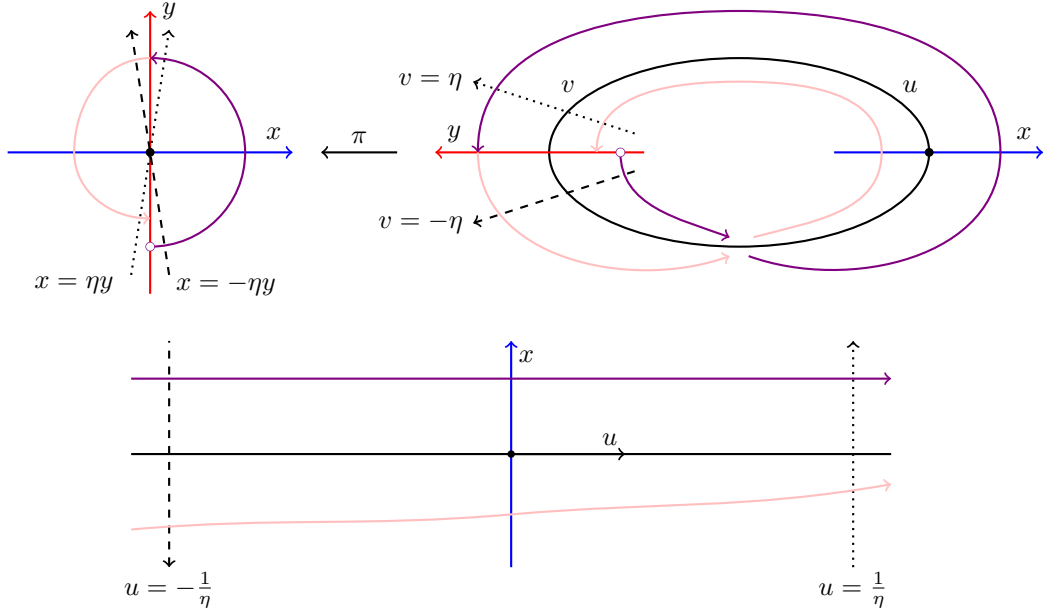
\begin{figure}[t]
\begin{center}
\begin{tikzpicture}
\begin{scope}[thick,xshift=-4.75cm,scale=1.25]
\draw[->,blue] (-1.5,0) to (1.5,0);
\draw[->,red] (0,-1.5) to (0,1.5);
\draw[fill] (0,0) circle [radius=1pt];
\draw[fill,violet] (0,-1) circle [radius=1.1pt];
\draw[->,violet] (0,-1) to [out=0,in=-90] (1,0) to [out=90,in=0] (0,1);
\draw[fill,white] (0,-1) circle [radius=.9pt];
\draw[->,pink] (0,1) to [out=180,in=90] (-.8,0) to [out=-90,in=180] (0,-.7);
\node at (1.3,.2) {$x$};
\node at (0.2,1.5) {$y$};
\draw[dotted,->] (-.2,-1.3) to (.2,1.3);
\draw[dashed,->] (.2,-1.3) to (-.2,1.3);
\node at (.8,-1.4) {$x=-\eta y$};
\node at (-.8,-1.4) {$x=\eta y$};
\end{scope}
\draw[thick,<-] (-2.5,0) to (-1.5,0);
\node at (-2,.2) {$\pi$};
\begin{scope}[thick, xshift=3cm,scale=1.25]
\draw (0,0) ellipse [x radius=2cm, y radius=1cm];
\draw[->,blue] (1,0) to (3.2,0);
\draw[->,red] (-1,0) to (-3.2,0);
\draw[fill] (2,0) circle [radius=1pt];
\draw[fill,violet] (-1.25,0) circle [radius=1.1pt];
\draw[->,violet] (-1.25,0) to [out=-90,in=160] (-0.1,-.9);
\draw[->,violet] (.1,-1.1) to [out=-20,in=-90] (2.75,0) to [out=90,in=0] (0,1.5) to [out=180,in=90] (-2.75,0);
\draw[fill,white] (-1.25,0) circle [radius=.9pt];
\draw[->,pink] (-2.75,0) to [out=-90,in=200] (-.1,-1.1);
\draw[->,pink] (.15,-.9) to [out=15,in=-90] (1.5,0) to [out=90,in=0] (0,.75) to [out=180,in=90] (-1.5,0);
\node at (3,.2) {$x$};
\node at (-3,.2) {$y$};
\node at (1.8,.7) {$u$};
\node at (-1.8,.7) {$v$};
\draw[->,dotted] (-1.1,.2) to (-2.8,.75);
\draw[->,dashed] (-1.1,-.2) to (-2.8,-.75);
\node at (-3.25,.75) {$v=\eta$};
\node at (-3.35,-.75) {$v=-\eta$};
\end{scope}
\begin{scope}[yshift=-4cm,thick]
\draw (-5,0) to (5,0);
\draw[->,blue] (0,-1.5) to (0,1.5);
\draw[->,dotted] (4.5,-1.5) to (4.5,1.5);
\draw[<-,dashed] (-4.5,-1.5) to (-4.5,1.5);
\draw[->] (0,0) to (1.5,0);
\node at (1.3,.2) {$u$};
\node at (0.2,1.3) {$x$};
\node at (-4.5,-1.8) {$u=-\frac{1}{\eta}$};
\node at (4.5,-1.8) {$u=\frac{1}{\eta}$};
\draw[fill] (0,0) circle [radius=1pt];
\draw[->,violet] (-5,1) to (5,1);
\draw[->,pink] (-5,-1) to [out=5,in=185] (0,-.8) to [out=5,in=190] (5,-.4);
\end{scope}
\end{tikzpicture}
\end{center}
\caption{Blowup of the origin for the family $Z_\mu$ in the charts $\pi(x,u)=(x,ux)$ and $\pi(v,y)=(vy,y)$.}\label{doble}
\end{figure}
Blowing up the origin, see Figure~\ref{doble}, we have that
\[Y_\mu=\frac{(x,ux)^*Z_\mu}{x^2}=(3\beta u^2+\beta u^4+ux+2x^2)\partial_u+(\beta u+\alpha u^2-\beta u^3-x)x\partial_x\]
and, assuming that $\beta>0$,
\[
X_\mu=\Big(\frac{x}{3\beta},\frac{y}{\sqrt{6\beta}}\Big)^*Y_\mu
=\left(x^2+y^2+\frac{xy}{\sqrt{6\beta}}+\frac{x^4}{27\beta^2}\right)\partial_x+\left(\frac{x}{3}-\frac{y}{\sqrt{6\beta}}+\frac{\alpha x^2}{9\beta^2}-\frac{x^3}{27\beta^2}\right)y\partial_y
\]
is of the form (\ref{X}) with
$f_1(x,y;\mu)=1+\frac{x^2}{27\beta^2}$, $f_2(x,y;\mu)\equiv 1$, $a(\mu)=\frac{1}{\sqrt{6\beta}}$, $g_1(x,y;\mu)=\frac{1}{3}+\frac{\alpha x}{9\beta^2}-\frac{x^2}{27\beta^2}$, $g_2(y;\mu)\equiv b(\mu)=-\frac{1}{\sqrt{6\beta}}$ and $c(\mu)=\frac{1}{3}$. 

Notice that $f_1(x,0;\mu)\neq 0$ for all $x\in\R$ and $d(\mu)=4(1-c(\mu))-(a(\mu)-b(\mu))^2=\frac{2}{3}(4-\frac{1}{\beta})$ is positive if and only if  $\beta>\frac{1}{4}$.
By applying Theorem~\ref{car} we also deduce that the origin is monodromic if and only if  $\beta>\frac{1}{4}$.
In that case we can apply Theorem~\ref{coef} and consider the limit (as $\eta\to0^+$) transition map $\Pi_{-\infty}^{+\infty}(y;\mu)=e^{\gamma_\pm(\mu)} y+\F_{1+\epsilon}^\infty(\{\beta>\frac{1}{4}\})$ of $X_\mu$, with
\[\gamma_\pm(\mu)=\mr{PV}\int_{-\infty}^{+\infty}\frac{\frac{1}{3}+\frac{\alpha x}{9\beta^2}-\frac{x^2}{27\beta^2}}{1+\frac{x^2}{27\beta^2}}\frac{dx}{x}\mp\frac{\pi}{\sqrt{4\beta-1}}=\pi\left(\frac{\alpha}{\beta\sqrt{3}}\mp\frac{1}{\sqrt{4\beta-1}}\right).\]
Assuming that $\beta>\frac{1}{4}$, the Poincar\'e return map $R(x;\mu)$ of $Z_\mu$ around $(0,0)$ is the composition of two transition maps  $\Pi_{-\infty}^{+\infty}(x;\mu)=e^{\gamma_\pm(\mu)}x+\F_{1+\epsilon}^\infty(\{\beta>\frac{1}{4}\})$ of $X_\mu$, one for $x>0$ and the other for $x<0$, see Figure~\ref{doble}, so that
$R(x;\mu)=e^{\gamma(\mu)} x+\F_{1+\epsilon}^\infty(\{\beta>\frac{1}{4}\})$ with $\gamma(\mu)=\gamma_+(\mu)+\gamma_-(\mu)=\frac{2\pi\alpha}{\beta\sqrt{3}}$, in agreement with \cite[p. 189]{GMM02}. 

Notice that for any $\beta>\frac{1}{4}$ the origin is a center of $Z_{(0,\beta)}$  because it is reversible via $(x,y,t)\mapsto (-x,y,-t)$.
Moreover, thanks to the uniform properties of the flat remainder class $\F_{1+\epsilon}^\infty$ we can assert that the ciclicity of the origin in the family (\ref{Z}) is  zero at any $\mu=(\alpha,\beta)\in\R^2$ with $\beta>\frac{1}{4}$. This last assertion can not be deduced from the results in \cite{GMM02} because they are not uniform with respect to parameters.

\bibliographystyle{plain}

\end{document}